\documentclass[10.8pt]{amsart}

\usepackage{amsmath,amssymb,amsfonts,amsthm}
\usepackage[a4paper, margin=1.15in]{geometry}
\usepackage{todonotes}
\usepackage{wrapfig}
\usepackage{graphicx, overpic}

\newtheorem*{theorem*}{Theorem}
\newtheorem*{conjecture*}{Conjecture}
\theoremstyle{definition}
\newtheorem*{intuition*}{Geometric intuition}
\newtheorem*{remark}{Remark}

\newcommand{\gen}[1]{\left< #1\right>}
\newcommand{\p}{\partial}
\newcommand{\Lam}{\Lambda}

\title{Surprising circles in Morse boundaries of right-angled Coxeter groups}
\author{Marius Graeber, Annette Karrer, Nir Lazarovich, and Emily Stark}
\date{\today}

\begin{document}

\thanks{The second author was supported by the Deutsche Forschungsgemeinschaft (DFG, German Research Foundation)--281869850 (RTG 2229). The third author was supported by the Israel Science Foundation (grant No. 1562/19) and the German-Israeli Foundation for Scientific Research and Development. The fourth author was supported by the Azrieli Foundation; partially supported at the Technion by a Zuckerman Fellowship; and partially supported by the NSF RTG grant $\#$1840190.}

\maketitle

\vskip.2in

\vspace{0.00001cm}
Given a graph $\Lambda$ the associated \emph{right-angled Coxeter group} is given by \[W_{\Lambda} = \gen{V(\Lambda) \, | \, \{v^2\}_{v\in V(\Lambda)} \cup  \{[v,w]\}_ {\{v,w\}\in E(\Lambda)}}.\]
See \cite{davis, dani} for background. 

A finitely generated group $\Gamma$ admits a quasi-isometry invariant {\it Morse boundary} defined by Cordes~\cite{cordes2017morse} and denoted $\p_M \Gamma$.
The Morse boundary of a CAT(0) group, such as a right-angled Coxeter group, is equal to the contracting boundary of the group, defined by Charney--Sultan~\cite{charney2015contracting}. Moreover, the Morse boundary of a hyperbolic group is equal to the visual boundary of the group. 

The topology of a boundary of a finitely generated group often captures algebraic information. 
For example, the visual boundary of a hyperbolic group is totally disconnected if and only if the group is virtually free. In contrast, the Morse boundary of every right-angled Artin group is totally disconnected~\cite[Theorem 1.1]{charneycordessisto}. The classification of the right-angled Coxeter groups that have totally disconnected Morse boundary is open. 
An induced cycle in a graph is \emph{burst} if it contains a pair of non-adjacent vertices that are contained in an induced $4$-cycle.

\begin{conjecture*}[Tran] \cite[Conjecture 1.14]{tran2019strongly} Let $\Lambda$ be a graph. The Morse boundary $\partial_M W_{\Lambda}$ is totally disconnected if and only if every induced cycle of length at least four in $\Lambda$ is burst.
\end{conjecture*}

We give a negative answer to the above conjecture.

\begin{theorem*} \label{thm:main}
    There exists a graph in which every induced cycle of length at least four is burst and the associated right-angled Coxeter group contains an embedded circle $\mathbb{S}^1$ in its Morse boundary.
\end{theorem*}

\begin{proof}
Let $\Lambda$ and $\Lambda'$ be the graphs in Figure \ref{fig: graphs}. Every cycle in $\Lambda$ is burst. The graph $\Lambda'$ is obtained by doubling $\Lambda$ over the star of $x$ then deleting the vertex $x$. Thus, the group $W_{\Lambda'}$ is an index-2 subgroup of $W_{\Lambda}$ by \cite[Lemma 2.3]{dani2014quasi}\footnote{The lemma does not appear in the published version~\cite{danithomasJSJ}}. Therefore, $\partial_M W_{\Lambda} = \partial_M W_{\Lambda'}$.
 The graph $\Lambda'$ contains a non-burst cycle, drawn in red. Thus, $\partial_M W_{\Lambda'}$, and hence $\partial_M W_{\Lambda}$, contains $\mathbb{S}^1$ by \cite[Corollary 1.12]{tran2019strongly}. This also follows from \cite[Theorem~1.4]{tran2019strongly} together with \cite[Proposition 4.9]{Genevois} or \cite[Theorem 7.5]{Tran_hierarchically}.
\end{proof}

\vskip.1in

    \begin{figure}[ht!]
    \centering
	\begin{overpic}[width=.6\textwidth,tics=5]{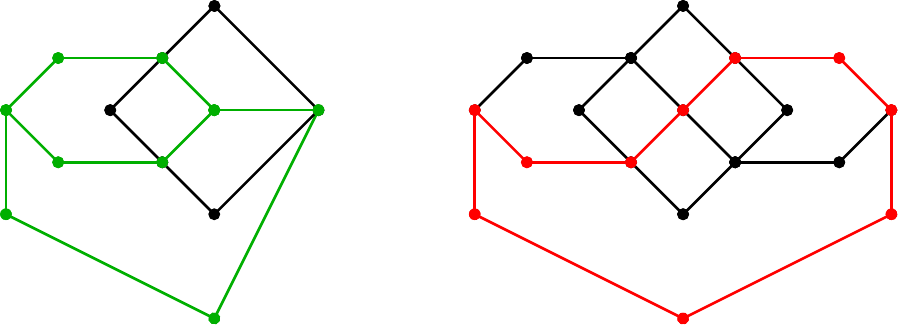} 
 	  \put(38,23){$x$}
    \end{overpic}
	\caption{The graphs $\Lambda$ on the left and $\Lambda'$ on the right.}
	\label{fig: graphs}
    \end{figure}

    \begin{figure}[ht!]
    \centering
	\begin{overpic}[width=.6\textwidth,tics=5]{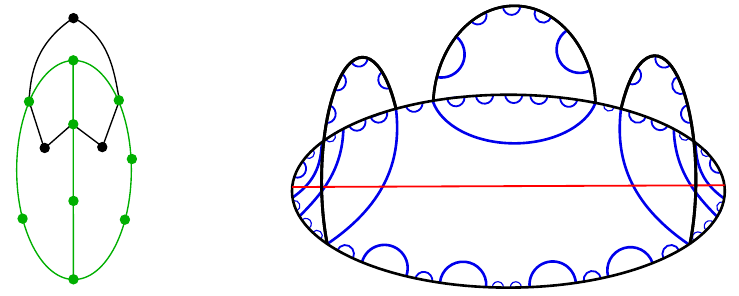} 
    \end{overpic}
	\caption{The graph $\Lambda$ and hyperbolic planes branching along lines.}
	\label{figure:planes}
    \end{figure}

\noindent {\bf Geometric intuition.}
    As shown above, the Morse boundary of the right-angled Coxeter group~$W_{\Lam}$ contains an embedded circle. Indeed, the group $W_{\Lam}$ has an index-two subgroup whose defining graph yields a cocompact Fuchsian reflection subgroup that intersects any Euclidean reflection subgroup in only a bounded set. 
    In particular, there is a quasi-isometrically embedded hyperbolic plane in the Davis complex for $W_{\Lam}$ that intersects any Euclidean plane in only a bounded subset; we sketch an alternative geometric description of such a subspace. 
        
    The graph $\Lambda$ contains a subdivided $\Theta$-graph $\Theta$ as an induced subgraph, as shown in green in the figures. The subgraph $\Theta$ contains three embedded induced cycles of length greater than four. These three cycles yield quasi-isometrically embedded hyperbolic planes in the Davis complex for~$W_{\Lambda}$. However, the Morse boundary does not contain the visual boundary of these hyperbolic planes. Indeed, each such cycle is burst, so each such hyperbolic plane in the Davis complex intersects a Euclidean plane in a line. 
    
    Nonetheless, a hyperbolic plane in the Davis complex for $W_{\Lam}$ that meets any Euclidean plane in only a bounded subset can be constructed by piecing together subsets of these hyperbolic planes and their $W_{\Theta}$-translates. The hyperbolic planes arising from the three cycles in $\Theta$ and their $W_{\Theta}$-translates intersect to form $\mathcal{X}$, a ``tree'' of hyperbolic planes branching along lines. See Figure~\ref{figure:planes} and \cite[Section 3]{hruskastarktran} for additional details. The cycles of length four in $\Lam$ yield Euclidean planes whose intersections with $\mathcal{X}$ are quasi-isometric to lines. These lines are contained in one hyperbolic plane and intersect branching lines transversely, as shown in red in Figure~\ref{figure:planes}. Thus, a hyperbolic plane meeting each Euclidean plane in a bounded subset can be constructed by making a choice at each branching line in $\mathcal{X}$. 
    
    Additional examples will appear in \cite[Section 5.5]{Karrer2020Thesis}.

\begin{remark}
    A {\it finite-index reflection subgroup} of a right-angled Coxeter group $W_{\Lambda}$ is a subgroup generated by reflections about the set of hyperplanes bounding a compact, convex subcomplex of the Davis complex for $W_{\Lambda}$. 
    The conjecture above still fails if one is allowed to pass to a finite-index reflection subgroup.
    Indeed, after an initial preprint of this paper, Hung Cong Tran explained if~$\Lambda$ is the $1$-skeleton of a $3$-cube, then $W_{\Lambda}$ also provides a counterexample to the conjecture above, using~\cite{charneycordessisto}. One can show there does not exist a finite-index reflection subgroup of $W_{\Lambda}$ whose defining graph contains an induced cycle of length at least four which is not burst. 
\end{remark}

\subsection*{Acknowledgements.} The authors are thankful for helpful discussions with Matthew Cordes, Pallavi Dani, Tobias Hartnick, Ivan Levcovitz, and Petra Schwer. The authors thank Hung Cong Tran for comments on a preprint of this paper. 

\bibliographystyle{abbrv}
\bibliography{biblio}

\end{document}